\documentclass[12pt, a4paper]{amsart}
\usepackage[utf8]{inputenc}
\usepackage[T1]{fontenc}
\usepackage{amsmath, amssymb, amsthm}
\usepackage[margin=1in]{geometry}
\usepackage{hyperref}
\usepackage{cite}
\usepackage{xcolor}
\usepackage{bm}
\usepackage{textcase}

\DeclareMathOperator{\diver}{div} 
\DeclareMathOperator{\vol}{Vol} 
\DeclareMathOperator{\tr}{tr}

\theoremstyle{plain}
\newtheorem{proposition}{Proposition}[section]
\newtheorem{cor}{Corollary}[section]
\newtheorem{theorem}{Theorem}[section]
\newtheorem{lemma}{Lemma}[section]
\newtheorem{ex}{Example}[section]

\theoremstyle{definition}

\newtheorem{remark}{Remark}

\title{Volume Growth of Self-Shrinkers of the
\NoCaseChange{$Q_k$}-Flow in \NoCaseChange{$\mathbb{R}^{n+1}$}}
\author{Junior Tavares}
\address{
Instituto de Matemática e Estatística,
Universidade Federal Fluminense,
São Domingos,
Niterói, RJ 24210-201, Brazil
}
\email{juniorts@id.uff.br}

\author{Detang Zhou}
\address{
Instituto de Matemática e Estatística,
Universidade Federal Fluminense,
São Domingos,
Niterói, RJ 24210-201, Brazil
}
\email{zhoud@id.uff.br}
\date{}
\thanks{J. Tavares was partially supported by CAPES/Brazil - Finance Code 001}
\thanks{D. Zhou was partially supported by CNPq/Brazil 308067/2023-1.
}

\begin{document}

\begin{abstract}
In this paper, we establish a general volume estimate for complete
Riemannian manifolds under suitable differential inequalities involving
a proper function and a symmetric tensor, without imposing curvature
assumptions. As an application, we prove that a class of orientable
hypersurfaces properly immersed in Euclidean space has at most
polynomial volume growth and finite weighted volume. This class includes,
in particular, certain self-shrinkers of the $Q_k$-flow and properly
immersed minimal hypersurfaces for which the normal component of the
position vector is bounded.
\end{abstract}

\maketitle
\section{Introduction}
The problem of estimating the volume growth of complete noncompact Riemannian manifolds is a classical topic in Differential Geometry and Geometric Analysis. A common approach to obtaining volume estimates is to impose suitable curvature conditions. The classical Bishop volume comparison theorem implies that if a complete noncompact $n$-dimensional Riemannian manifold has nonnegative Ricci curvature, then its volume growth is at most Euclidean. On the other hand, Calabi~\cite{calabi1975} and Yau~\cite{MR417452} proved that a complete noncompact Riemannian manifold with nonnegative Ricci curvature has at least linear volume growth.

For gradient shrinking Ricci solitons, Cao and Zhou~\cite{MR2732975} proved that every complete noncompact gradient shrinking Ricci soliton has at most Euclidean volume growth.
Ding and Xin~\cite{MR3119795} proved that every complete noncompact properly immersed self-shrinker in $\mathbb{R}^{n+m}$ has at most Euclidean volume growth.
Cheng and Zhou~\cite{MR2996973} established volume estimates without imposing curvature assumptions and applied their results to self-shrinkers of the mean curvature flow. In particular, they proved that, for a complete self-shrinker $\Sigma^n$ immersed in $\mathbb{R}^{n+1}$, the following conditions are equivalent: the immersion is proper, $\Sigma$ has polynomial volume growth, $\Sigma$ has Euclidean volume growth, and $\Sigma$ has finite weighted volume.

Later, Cheng, Vieira, and Zhou~\cite{MR4300234} established another volume estimate without curvature assumptions and applied it to complete submanifolds in gradient shrinking Ricci solitons with bounded weighted mean curvature. More precisely, they proved that if $\Sigma$ is a complete submanifold immersed in a complete gradient shrinking Ricci soliton $(M,\overline{g},f)$, the weighted mean curvature vector of $\Sigma$ is bounded, and $\operatorname{tr}_{\Sigma^\perp}\overline{\nabla}^2f \geq \frac{k}{2}$ for some constant $k$, then polynomial volume growth, properness of the immersion, and finite weighted volume are equivalent. Moreover, if the ambient shrinking Ricci soliton has bounded geometry, then every complete, noncompact, properly immersed submanifold with bounded weighted mean curvature has at least linear volume growth.  This extends an earlier result of Wei~\cite{MR3658212}, obtained for $f$-minimal submanifolds.

In this context, the main purpose of this paper is to establish a volume estimate for complete Riemannian
manifolds involving a general divergence-form operator associated with a symmetric bundle
endomorphism. This estimate extends the theorem of Cheng and Zhou from the Laplace operator to a
broad class of second-order operators and yields applications to hypersurfaces whose geometry is
governed by Newton transformations. As consequences, we obtain weighted volume estimates,
polynomial volume growth, and half-space type theorems for self-shrinkers of the $Q_k$-flow,
recovering the classical results for mean curvature flow as the particular case $k=0$.

Along this line, and using ideas from~\cite{MR2996973, MR4300234}, we prove the following result.

\begin{theorem}\label{Tp1}
Let $M$ be a complete Riemannian manifold, $P: TM \to TM$ a smooth symmetric bundle endomorphism, and let $f \in C^2(M)$ be a proper nonnegative function.  
Assume there exist constants $\alpha, \beta > 0$ and $a\ge 0$ such that
\begin{align*}
\diver\bigl(P(\nabla f)\bigr) - \alpha \langle P(\nabla f), \nabla f\rangle + \beta f &\leq a,\\
\diver\bigl(P(\nabla f)\bigr) &\leq a,\\
\langle P(\nabla f), \nabla f\rangle &\geq 0.
\end{align*}
Then
\begin{enumerate}
\item[(i)]
\[
\int_M e^{-\alpha f}\, dv_g < +\infty.
\]

\item[(ii)]
For $r \geq 1$, the volume of $D_r$ has polynomial growth:
\[
\vol(D_r) \leq C r^{\frac{2\alpha a}{\beta}},
\]
where
\(
D_r = \left\{ x \in M: 2\sqrt{f}\leq \sqrt{2\beta}r \right\}
\)
and
\(
C = e^{\frac{\alpha \beta}{2}} \int_M e^{-\alpha f}\, dv_g.
\)
\end{enumerate}
\end{theorem}

\begin{remark}
Theorem~\ref{Tp1} extends Theorem~1.1 of~\cite{MR4300234} from the case $P=I$ to a general smooth symmetric bundle endomorphism $P$. Consequently, Theorem~1.1 of~\cite{MR2996973} is also recovered as a particular case.
\end{remark}

\begin{remark}
Under the first differential inequality, the condition $\langle P(\nabla f),\nabla f\rangle
\le \frac{\beta}{\alpha}f$ implies $\operatorname{div}\bigl(P(\nabla f)\bigr)\le a$.
Indeed,
\[
\operatorname{div}\bigl(P(\nabla f)\bigr)
\le
a+\alpha\langle P(\nabla f),\nabla f\rangle-\beta f
\le a.
\]
Thus, the former condition may be used as a sufficient condition for
the second differential inequality in Theorem~\ref{Tp1}.
\end{remark}

The polynomial volume growth condition plays an important role in the study of self-shrinkers and related classes of hypersurfaces. In the setting of mean curvature flow, it appears in the work of Colding and Minicozzi \cite{MR2993752} and Ecker and Huisken \cite{MR1025164}, and it has also been used in several rigidity and classification results for self-shrinkers and $\lambda$-hypersurfaces; see, for instance, Cao and Li \cite{MR3018176}, Cheng, Ogata, and Wei \cite{MR3514553}, and Cheng and Wei \cite{MR3763110}. More recently, Ancari and Miranda \cite{MR4236542} showed that, in certain classification results for $\lambda$-hypersurfaces, polynomial volume growth can be replaced by suitable intrinsic assumptions. These results further motivate the search for geometric conditions ensuring polynomial volume growth.

We next apply our volume growth estimate to self-shrinkers of the $Q_k$-flow, that is, hypersurfaces satisfying
\[
Q_k := \frac{\sigma_{k+1}}{\sigma_k} = \frac{1}{2(k+1)} \langle x, \nu \rangle,
\]
where $\sigma_k$ denotes the $k$-th elementary symmetric function of the principal curvatures, as defined in \eqref{sigmak}, and $\nu$ is a chosen unit normal vector field. 

When $k=0$, the equation reduces to the classical self-shrinker equation for the mean curvature flow. For $k\in\{1, \cdots, n-1\}$, the sphere $\mathbb{S}^n(\sqrt{2(n-k)})$ and the cylinder $\mathbb{S}^m(0,\sqrt{2(m-k)})\times\mathbb{R}^{n-m}$ provide examples of self-shrinkers of the $Q_k$-flow whenever $k\leq m-1$.

Curvature flows whose normal velocity is given by a nonlinear symmetric function of the principal curvatures have attracted considerable attention. Important examples include the Gauss curvature flow, the $Q_k$-flow, and the harmonic mean curvature flow. These flows have been extensively investigated by Andrews \cite{MR1714339,MR1385524,MR2339467}, Choi and Daskalopoulos \cite{MR4380033}, Dieter \cite{MR2106768}, among others, due to their deep connections with geometric analysis, nonlinear partial differential equations, and the geometry of convex hypersurfaces.

A fundamental tool in the study of higher-order mean curvatures is provided by the Newton transformations. These tensors play a considerable role in the theory of hypersurfaces and have been extensively employed in the study of stability problems, Minkowski-type formulas, integral identities, rigidity theorems, and fully nonlinear curvature flows (see, for instance, Reilly \cite{MR341351} and Barbosa and Colares \cite{MR1456513} and the references therein).

From an analytic viewpoint, the differential operator 
\[
L_k u=\operatorname{div}(P_k(\nabla u)),
\]
defined on a hypersurface immersed in $\mathbb{R}^{n+1}$ and associated with the $k$-th Newton transformation $P_k$, is the natural analogue of the Laplace operator in the setting of higher-order mean curvatures. Indeed, when $k=0$, one has $P_0=I$ and $L_0=\Delta$. Likewise, the weighted operator associated with a smooth function
$f$ defined on $\mathbb{R}^{n+1}$
\[
L_{k,f} u=e^{f}\operatorname{div}\left(e^{-f}P_k(\nabla u)\right)
\]
generalizes the classical drifted Laplacian, which plays a fundamental role in the analysis of self-shrinkers of the mean curvature flow.

Motivated by the mean curvature flow and by the weighted mean curvature vector, we define the weighted $Q_k$-curvature vector by
\[
\bm{H}_{k,f} = \sigma_k (\overline{\nabla} f)^\perp - (k+1)\sigma_{k+1}\nu,
\]
and the corresponding scalar quantity
\[
H_{k,f} = \langle \bm{H}_{k,f}, \nu \rangle.
\]


As an application of Theorem \ref{Tp1}, we obtain the following volume estimate for hypersurfaces, which may be of independent interest.
\begin{theorem}\label{Tp2}
Let $\Sigma^n$ be an orientable hypersurface properly immersed in $\mathbb{R}^{n+1}$ such that $\big|(\overline{\nabla}f)^\perp\big||\bm{H}_{k,f}|$ is bounded, where $f=\frac{|x|^2}{4}$.
Assume that $\sigma_k$ is bounded and that $P_k$ is coercive.
 Then 
\[
\displaystyle\int_{\Sigma} e^{-f} dv_g < +\infty
\]
and $\Sigma$ has polynomial volume growth. More precisely, there exist constants $\beta>0$ and $a \geq 0$ such that, for $r \geq \sqrt{2\beta}$,
\[
\vol(B_r(0) \cap \Sigma) \le C r^{\frac{2 a}{\beta}},
\]
where $C = e^{\frac{\beta}{2}}
(2\beta)^{-\frac{a}{\beta}}
\int_{\Sigma} e^{-f}\,dv_g$.
\end{theorem}

Theorem \ref{Tp2} recovers the Euclidean volume growth estimate of Theorem 1.2 of Cheng and Zhou \cite{MR2996973} in the case of self-shrinkers of the mean curvature flow. Indeed, taking $k=0$, we immediately obtain $\mathbf{H}_{0,f} = \mathbf{H}_f$, $\sigma_0 = 1 > 0$, and $P_0 = I$, which is coercive.

A consequence of Theorem \ref{Tp2} arises in the context of minimal hypersurfaces. Taking $k=0$, we obtain
\[
\left| \left( \overline{\nabla} f \right)^\perp \right| \left| \mathbf{H}_{0,f} \right|
= \frac{1}{2} |x^\perp| \left| \frac{1}{2}x^\perp + \mathbf{H} \right|
= \frac{|x^\perp|^2}{4}.
\]
Therefore, we obtain the following corollary for minimal hypersurfaces.

\begin{cor}
Let $\Sigma^n$ be an orientable minimal hypersurface properly immersed in $\mathbb{R}^{n+1}$. Assume that $\bigl|x^\perp\bigr|$ is bounded. Then $\Sigma$ has polynomial volume growth. 
\end{cor}

\begin{remark}\label{rem:comparison-growth-conditions}
It is worth comparing the previous corollary with Theorem~0.4 of Colding and Minicozzi
in \cite{ColdingMinicozzi2026}, which states that a complete,
noncompact, properly immersed stationary varifold with sublinear height
growth has Euclidean volume growth. In the hypersurface setting, the
sublinear-height condition means that, after a suitable choice of
coordinates, there exist $\alpha<1$ and $r_0>0$ such that
\[
    |x_{n+1}|
    \leq
    \bigl(x_1^2+\cdots+x_n^2\bigr)^{\alpha/2}
\]
for every $x\in\Sigma\setminus B_{r_0}$. Under this assumption,
Theorem~0.4 in \cite{ColdingMinicozzi2026} yields the sharper estimate
\[
    \operatorname{Vol}\bigl(\Sigma\cap B_r\bigr)
    \leq Cr^n.
\]

The hypothesis in the previous corollary, namely the boundedness of the
normal component of the position vector,
\[
    |x^\perp|\leq C,
\]
is of a rather different geometric nature. 
In general, neither of these two conditions implies the other, even
within the class of complete, properly embedded minimal hypersurfaces.
Examples~\ref{exH} and~\ref{cat}, discussed in detail in
Section~\ref{SimonsCone}, illustrate the two possible failures of
implication.
\end{remark}

We now apply Theorem \ref{Tp2} to obtain a half-space type result. Such theorems play an important role in Differential Geometry and, in particular, in Geometric Analysis. Beyond their intrinsic interest, they often provide powerful tools for deriving rigidity and classification results for hypersurfaces in $\mathbb{R}^{n+1}$.

In the setting of minimal surfaces, Hoffman and Meeks proved in \cite{MR1062966} that a connected, proper, possibly branched, nonplanar minimal surface in $\mathbb{R}^{3}$ cannot be contained in a half-space. More recently, in the context of translating solitons of the $r$-mean curvature flow, Alencar, Bessa, and Neto established analogous results in \cite{MR5017623}, showing that properly immersed translating solitons cannot be confined to certain half-spaces opposite to the translation direction. More generally, they obtained nonexistence results for complete translating solitons of the $r$-mean curvature flow under suitable growth assumptions on the $(r-1)$-mean curvature and on the norm of the second fundamental form.

Before stating the next theorem, which may be regarded as a half-space theorem for self-shrinkers of the $Q_k$-flow, we adopt the following convention to simplify the notation.
Since $P_k = \sigma_k I - A P_{k-1}$ for $1 \leq k \leq n$ and $P_0 = I$, we may define $P_{-1} = 0$ so that the definition of $P_k$ remains consistent for all $k \geq 0$. We adopt this convention to simplify the notation, making it clear that the following theorem generalizes the classical half-space theorem for self-shrinkers of the mean curvature flow (see Theorems 1, 3, and 1.1 in \cite{MR4724121}, \cite{MR3152087}, and \cite{MR3483061}, respectively).

\begin{theorem}\label{hiperplan}
Let $\Sigma^n$ be an orientable self-shrinker of the $Q_k$-flow properly immersed in $\mathbb{R}^{n+1}$. Suppose that $\sigma_k$ is bounded and that $P_k$ is coercive. Let $p\in\mathbb{R}^{n+1}\setminus\{0\}$. If
\[
\int_{\Sigma}
\left\langle
\nabla\left(\frac{\sigma_{k+1}}{\sigma_k}\right),
P_{k-1}(p^\top)
\right\rangle e^{-f}\,dv_g
\leq 0,
\]
then either $\Sigma\not\subset
\left\{
x\in\mathbb{R}^{n+1}:\langle x,p\rangle\geq0
\right\}$, or $k=0$ and $\Sigma
=
\left\{
x\in\mathbb{R}^{n+1}:\langle x,p\rangle=0
\right\}$.

Similarly, if
\[
\int_{\Sigma}
\left\langle
\nabla\left(\frac{\sigma_{k+1}}{\sigma_k}\right),
P_{k-1}(p^\top)
\right\rangle e^{-f}\,dv_g
\geq 0,
\]
then either $\Sigma\not\subset
\left\{
x\in\mathbb{R}^{n+1}:\langle x,p\rangle\leq0
\right\}$, or $k=0$ and $\Sigma
=
\left\{
x\in\mathbb{R}^{n+1}:\langle x,p\rangle=0
\right\}$.

In particular, if $Q_k$ is constant and $k\geq1$, then $\Sigma$ cannot
be contained in either closed half-space determined by a hyperplane
through the origin. If $k=0$, then a self-shrinker contained in one of
these closed half-spaces must coincide with its boundary hyperplane.
\end{theorem}

As a further consequence of Theorem~\ref{Tp2}, we obtain the following rigidity result.
\begin{theorem}\label{sphere}
Let $\Sigma^n$ be an orientable self-shrinker of
the $Q_k$-flow properly immersed in $\mathbb{R}^{n+1}$. Suppose that
$\sigma_k$ is bounded and that $P_k$ is coercive. If
\[
|x^\perp|^2\geq 2(n-k),
\]
then $\Sigma =\mathbb{S}^n(\sqrt{2(n-k)})$.
\end{theorem}

The paper is organized as follows. Section \ref{pre} contains the preliminaries.
In Section \ref{proofTp1}, we prove Theorem \ref{Tp1}. Section \ref{s4}
is devoted to the proof of Theorem \ref{Tp2} and to its consequences.
More precisely, in Subsection \ref{sai}, we establish some auxiliary
computations that will be used in the proof, while in Subsection
\ref{proofTp2} we prove Theorem \ref{Tp2} and derive several applications. Finally, in Section~\ref{SimonsCone},
we discuss examples showing that boundedness of $|x^\perp|$ and
sublinear height growth are, in general, independent conditions.

\section{Preliminaries}\label{pre}

Let $x:\Sigma^n \to \mathbb{R}^{n+1}$ be a connected, oriented, properly immersed hypersurface.  
We denote by $\overline{\nabla}$ the Levi-Civita connection of $\mathbb{R}^{n+1}$ and by $\nabla$ the induced connection on $\Sigma$. We shall make no distinction in notation between a function
$f:\mathbb{R}^{n+1}\to\mathbb{R}$ and its restriction to $\Sigma$. Moreover, if $f$ is proper and $\Sigma$ is properly
immersed in $\mathbb{R}^{n+1}$, then the restriction of $f$ to $\Sigma$
is also proper.

The shape operator $A$ is given by
\[
A(X) = \overline{\nabla}_X \nu.
\]

The principal curvatures $\lambda_1,\ldots,\lambda_n$ are the eigenvalues of $A$.  
For $1\leq k\leq n$, the $k$-th elementary symmetric function of the principal curvatures is defined by
\begin{align}\label{sigmak}
\sigma_k
=
\sum_{i_1<\cdots<i_k}
\lambda_{i_1}\cdots\lambda_{i_k},
\end{align}
with the conventions $\sigma_0=1$ and $\sigma_j=0$ for $j>n$.

The $k$-th Newton transformation is defined recursively by
\[
P_0 = I, \qquad P_k = \sigma_k I - A P_{k-1}.
\]
For convenience, we also set $P_{-1}=0$.

If $Ae_i=\lambda_i e_i$, denote by $A_i$ the restriction of $A$ to the orthogonal complement
of $e_i$. Hence,
\[
\sigma_k(A_i)
=
\sum_{\substack{1\leq i_1<\cdots<i_k\leq n\\ i \notin\{i_1,\ldots,i_k\}}}
\lambda_{i_1}\cdots\lambda_{i_k}.
\]

The following lemma provides some useful relations between the elementary symmetric functions and the Newton transformations.
\begin{lemma}[\cite{MR1456513}, Lemma 2.1]\label{l1bc}
For $0\le k\le n-1$, we have:
\begin{enumerate}
\item[(a)] $\tr(P_k) = (n-k)\sigma_k$;
\item[(b)] $\tr(AP_k) = (k+1)\sigma_{k+1}$;
\item[(c)] $\tr(A^2 P_k) = \sigma_1\sigma_{k+1} - (k+2)\sigma_{k+2}$;
\item[(d)]
If $Ae_i=\lambda_i e_i$, then $P_k e_i=\sigma_k(A_i)e_i$.
\end{enumerate}
\end{lemma}

A symmetric endomorphism $P:T\Sigma \to T\Sigma$ is said to be coercive if there exists a constant $\delta>0$ such that
\[
\langle P(X),X\rangle \geq \delta |X|^2
\]
for every $X\in T\Sigma$.

Equivalently, $P$ is uniformly positive definite, that is, all eigenvalues of $P$ are bounded from below by a positive constant.

\begin{remark}\label{cisk}
For $0\leq k\leq n-1$, the coercivity of $P_k$ implies that there exists a constant $\delta_0>0$ such that $\sigma_k\geq\delta_0$. Indeed, by item~(a) of Lemma \ref{l1bc},
\[
\sigma_k=\frac{\operatorname{tr}(P_k)}{n-k}
\geq \frac{n\delta}{n-k}\geq \delta >0,
\]
where $\delta>0$ is a lower bound for the eigenvalues of $P_k$.
\end{remark}

Note that
\begin{align*}
\nonumber \diver(P_k(W)) &= \sum_{i=1}^n \langle \nabla_{e_i}(P_k(W)), e_i \rangle \\
\nonumber &= \sum_{i=1}^n \langle \left(\nabla_{e_i} P_k \right)(W),e_i \rangle + \sum_{i=1}^n \langle P_k(\nabla_{e_i}W), e_i \rangle \\
\nonumber &= \langle \diver P_k, W\rangle + \tr\left( Y \to P_k(\nabla_Y W)\right)\\
&= \langle \diver P_k, W \rangle + \tr(P_k \circ \nabla W). 
\end{align*}
 It is well known that the Newton transformation $P_k$ is divergence-free for hypersurfaces immersed in $\mathbb{R}^{n+1}$; see, for instance, \cite{MR341351, MR1216008}. Hence,
\begin{align}\label{eqdivtr}
\operatorname{div}\!\left(P_k(\nabla u)\right)
=
\operatorname{tr}\!\left(P_k\circ\nabla^2u\right),
\end{align}
for every $u\in C^2(\Sigma)$.

Define the operator
\[
L_k u = \operatorname{div}(P_k(\nabla u)).
\]
 By \eqref{eqdivtr},
\[
L_k u=\operatorname{tr}\!\left(P_k\circ\nabla^2u\right).
\]
 
Since $\mathbb{R}^{n+1}$ has zero sectional curvature, we have the following special case of Lemma 5.2 in \cite{MR1456513}:

\begin{lemma}[\cite{MR1456513}, Lemma 5.2]\label{lema52}
For any immersion $x: \Sigma^n \to \mathbb{R}^{n+1}$ and for any fixed vector $a \in \mathbb{R}^{n+1}$, we have
\[
L_k \phi_a = -(k+1)\sigma_{k+1}\varphi_a.
\]
Here $\varphi_a = \langle a, \nu \rangle$ and $\phi_a = \langle a, x \rangle$.
\end{lemma}

We also define the weighted operator
\[
L_{k,f}u = e^{f}\operatorname{div}(e^{-f}P_k(\nabla u))
= L_k u - \langle \nabla f, P_k(\nabla u)\rangle.
\]

Note that $L_0$ is the Laplacian, while $L_{0,f}$ denotes the drifted Laplacian. 

\subsection{The \texorpdfstring{$Q_k$}{Qk}-flow}

For the sake of completeness, we briefly review the $Q_k$-flow in this subsection.

Let $x : \Sigma^n \to \mathbb{R}^{n+1}$ be a properly immersed hypersurface.  
Consider a family of immersions $F : \Sigma \times [0,T) \to \mathbb{R}^{n+1}$ satisfying
\begin{align}\label{eq:flow}
\begin{cases}
\dfrac{\partial F(p,t)}{\partial t}
= -(k+1)Q_k(\lambda_1(p,t),\ldots,\lambda_n(p,t))\,\nu(p,t),
& (p,t)\in \Sigma\times[0,T),\\[1ex]
F(p,0) = x(p), & p\in \Sigma,
\end{cases}
\end{align}
where $\lambda_1,\ldots,\lambda_n$ are the principal curvatures of $x_t(\Sigma)$, $\nu$ is a chosen unit normal vector field and
\[
Q_k = \dfrac{\sigma_{k+1}}{\sigma_k}.
\]
The flow defined by \eqref{eq:flow} will be referred to as the
$Q_k$-flow. Since the factor $k+1$ is constant, it differs from the
usual $Q_k$-flow only by a reparametrization of time.

\begin{remark}
Our indexing is shifted by one with respect to the standard notation in
the literature on the $Q_k$-flow: our quotient
\[
Q_k=\frac{\sigma_{k+1}}{\sigma_k}
\]
corresponds to $Q_{k+1}$ in that notation. Moreover, we use the
time-rescaled flow
\[
\frac{\partial F}{\partial t}=-(k+1)Q_k\nu.
\]
The factor $k+1$ amounts only to a constant reparametrization of time
and is chosen so that the self-shrinker equation takes the form
\[
Q_k=\frac{1}{2(k+1)}\langle x,\nu\rangle.
\]
\end{remark}

A hypersurface $x:\Sigma^n\to\mathbb{R}^{n+1}$, with $\sigma_k\neq 0$,
is called a \emph{self-shrinker of the $Q_k$-flow} if it satisfies
\[
Q_k
=
\frac{\sigma_{k+1}}{\sigma_k}
=
\frac{1}{2(k+1)}\langle x,\nu\rangle.
\]
Throughout this paper, we adopt this normalization of the self-shrinker
equation. In particular, when $k=0$, since $\sigma_0=1$ and
$\sigma_1=H$, the equation reduces to
\[
H=\frac12\langle x,\nu\rangle,
\]
which is the classical self-shrinker equation for the mean curvature flow.

Andrews~\cite{MR1385524, MR2339467} proved the existence of such flows under convexity assumptions and showed that the hypersurfaces converge in finite time to a round sphere. 
In particular, he established the existence of solutions to the $Q_k$-flow for hypersurfaces with positive principal curvatures.

Dieter \cite{MR2106768} studied the convergence of the $Q_k$-flow and proved short-time existence for weakly convex hypersurfaces satisfying $\sigma_k>0$. 
Note that, when $k=0$, we recover the mean curvature flow. Another case
of particular interest is the harmonic mean curvature flow, which,
up to the constant time rescaling adopted in \eqref{eq:flow},
corresponds to $k=n-1$.
Daskalopoulos and Sesum~\cite{MR2564404} studied the harmonic mean
curvature flow of compact star-shaped mean-convex surfaces in
$\mathbb{R}^3$, obtaining finite-time extinction and, under additional
assumptions, asymptotic sphericity results, in particular in the
rotationally symmetric case.

Choi and Daskalopoulos~\cite{MR4380033} established long-time existence
results for complete noncompact weakly convex hypersurfaces evolving by
the $Q_k$-flow and characterized the maximal existence time in terms of
the dimension of the space of unbounded directions of the initial
hypersurface.

Caputo, Daskalopoulos and Sesum established in \cite{MR2748631} the existence and uniqueness of a solution of the $Q_k$-flow in the viscosity sense for compact convex hypersurfaces. Moreover, they showed that, for compact convex hypersurfaces with flat sides, if the initial hypersurface satisfies a suitable nondegeneracy condition, then the interface separating the flat region from the strictly convex region becomes smooth and evolves by the $Q_{k-1}$-flow, at least for a short time.

We conclude this section by identifying the self-shrinking generalized cylinders for the $Q_k$-flow.
\begin{lemma}
Let $1\leq m\leq n$, $a\in\mathbb R^{m+1}$,
$0\leq k\leq m-1$, and $\rho>0$. The hypersurface
$\mathbb{S}^m(a,\rho)\times\mathbb{R}^{n-m}$ immersed in $\mathbb{R}^{n+1}$
is a self-shrinker of the $Q_k$-flow if and only if it coincides with the product
$\mathbb{S}^m(0,\sqrt{2(m-k)})\times\mathbb{R}^{n-m}$.
\end{lemma}
\begin{proof}
Let $\{E_1,\ldots,E_{m+1}\}$ be the canonical basis of
$\mathbb R^{m+1}$, and let
$(x_1,\ldots,x_{m+1})$ denote the corresponding Cartesian coordinates.
The outward unit normal vector field along
$\mathbb S^m(a,\rho)\times\mathbb R^{n-m}$ is
\[
\nu=\frac{1}{\rho}\sum_{i=1}^{m+1}(x_i-a_i)E_i.
\]
The principal curvatures are
\[
\lambda_1=\cdots=\lambda_m=\frac{1}{\rho},
\qquad
\lambda_{m+1}=\cdots=\lambda_n=0.
\]
Hence,
\[
\sigma_k=
\begin{cases}
\displaystyle \binom{m}{k}\rho^{-k}, & k\le m,\\[1ex]
0, & k>m,
\end{cases}
\]
and, for $k\le m-1$,
\[
\frac{\sigma_{k+1}}{\sigma_k}
=
\frac{\binom{m}{k+1}}{\binom{m}{k}}
\frac{\rho^{-k-1}}{\rho^{-k}}
=
\frac{m-k}{k+1}\frac1\rho.
\]
By the self-shrinker equation,
\[
\frac{\sigma_{k+1}}{\sigma_k}
=
\frac{1}{2(k+1)}\langle x,\nu\rangle.
\]
Therefore,
\[
\frac{m-k}{k+1}\frac{1}{\rho}
=
\frac{1}{2(k+1)}
\frac{1}{\rho}
\sum_{i=1}^{m+1}x_i(x_i-a_i),
\]
which simplifies to
\[
\sum_{i=1}^{m+1}x_i(x_i-a_i)=2(m-k).
\]
On the other hand,
\[
\sum_{i=1}^{m+1}(x_i-a_i)^2=\rho^2.
\]
Subtracting these two identities yields
\[
\sum_{i=1}^{m+1}a_i(x_i-a_i)
=
2(m-k)-\rho^2.
\]
The right-hand side is constant, whereas
\[
\sum_{i=1}^{m+1}a_i(x_i-a_i)
=
\langle a,x-a\rangle
\]
is a linear function on the sphere
$\mathbb S^m(a,\rho)$.
Since $x-a$ ranges over all vectors of length $\rho$, this function is
constant if and only if $a=0$.

Consequently, the cylinder is centered at the origin. Then
\[
\sum_{i=1}^{m+1}x_i^2=\rho^2,
\]
and the identity
\[
\sum_{i=1}^{m+1}x_i(x_i-a_i)=2(m-k)
\]
reduces to
\[
\rho^2=2(m-k).
\]
Hence,
\[
\rho=\sqrt{2(m-k)},
\]
and therefore
\[
\mathbb S^m(a,\rho)\times\mathbb R^{n-m}
=
\mathbb S^m(0,\sqrt{2(m-k)})\times\mathbb R^{n-m}.
\]

The converse is a direct computation.
\end{proof}

\section{Proof of Theorem \ref{Tp1}}\label{proofTp1}

In this section, we prove a general volume estimate result based on
\cite{MR2996973} and \cite{MR4300234}.

\begin{theorem}[Theorem \ref{Tp1}]
Let $M$ be a complete Riemannian manifold, $P: TM \to TM$ a smooth symmetric bundle endomorphism, and let $f \in C^2(M)$ be a proper nonnegative function.  
Assume there exist constants $\alpha, \beta > 0$ and $a\ge 0$ such that
\begin{align*}
\diver\bigl(P(\nabla f)\bigr) - \alpha \langle P(\nabla f), \nabla f\rangle + \beta f &\leq a,\\
\diver\bigl(P(\nabla f)\bigr) &\leq a,\\
\langle P(\nabla f), \nabla f\rangle &\geq 0.
\end{align*}
Then
\begin{enumerate}
\item[(i)]
\[
\int_M e^{-\alpha f}\, dv_g < +\infty.
\]

\item[(ii)]
For $r \geq 1$, the volume of $D_r$ has polynomial growth:
\[
\vol(D_r) \leq C r^{\frac{2\alpha a}{\beta}},
\]
where
\(
D_r = \left\{ x \in M: 2\sqrt{f}\leq \sqrt{2\beta}r \right\}
\)
and
\(
C = e^{\frac{\alpha \beta}{2}} \int_M e^{-\alpha f}\, dv_g.
\)
\end{enumerate}
\end{theorem}

\begin{proof}
Define
\[
L_{\phi_t} u := e^{\phi_t}\operatorname{div}(e^{-\phi_t} P(\nabla u)),
\]
where
\[
\phi_t = a\log t + \frac{\alpha}{t^\gamma} f
\quad \text{and} \quad
\gamma = \frac{\beta}{\alpha}.
\]

Note that, for $t \geq 1$,
\begin{align*}
(L_{\phi_t} f)e^{-\phi_t} + t\,\frac{d}{dt} e^{-\phi_t}
&= e^{-\phi_t}
\big( \operatorname{div}(P(\nabla f))
      - \langle P(\nabla f), \nabla \phi_t \rangle \big)
   + t\,\frac{d}{dt} e^{-\phi_t} \\
&= \big( \operatorname{div}(P(\nabla f))
        - \frac{\alpha}{t^\gamma}\langle P(\nabla f), \nabla f \rangle
        - a + \frac{\beta}{t^\gamma} f \big)e^{-\phi_t} \\
&= \Bigg[
\frac{1}{t^\gamma}
\big( \operatorname{div}(P(\nabla f))
      - \alpha \langle P(\nabla f), \nabla f \rangle
      + \beta f - a \big) \\
&\hspace{1cm}
+ \left(1 - \frac{1}{t^\gamma}\right)
(\operatorname{div}(P(\nabla f)) - a)
\Bigg] e^{-\phi_t} \\
&\leq 0.
\end{align*}
Hence, at a regular value of $f$ and for $t \ge 1$, we obtain
\[
\begin{aligned}
\frac{d}{dt} \int_{D_r} e^{-\phi_t}
&= \int_{D_r} \frac{d}{dt} e^{-\phi_t} \\
&\le -\frac{1}{t} \int_{D_r} (L_{\phi_t} f)e^{-\phi_t}
= -\frac{1}{t} \int_{D_r}
\operatorname{div}(e^{-\phi_t} P(\nabla f)) \\
&= -\frac{1}{t} \int_{\partial D_r}
\frac{1}{|\nabla f|}
\langle P(\nabla f), \nabla f \rangle e^{-\phi_t} \\
&\le 0,
\end{aligned}
\]
since $\langle P(\nabla f), \nabla f \rangle \ge 0$.

Integrating from $1$ to $T$, we obtain
\[
\int_{D_r} e^{-\phi_T}
\le \int_{D_r} e^{-\phi_1}
= \int_{D_r} e^{-\alpha f},
\]
because $\phi_1 = \alpha f$.

Taking $T = r^{\frac{2}{\gamma}}$, we obtain
\begin{align}\label{d1p0}
r^{-\frac{2a}{\gamma}}
\int_{D_r} e^{-\frac{\alpha}{r^2} f}
\le \int_{D_r} e^{-\alpha f}.
\end{align}
Since the integrals above are right–continuous in $r$, the inequality holds for all $r \ge 1$.
Moreover, since $f \le \frac{\beta}{2} r^2$ on $D_r$, it follows that
\begin{align}\label{d2p0}
\int_{D_r} e^{-\frac{\alpha}{r^2} f}
\ge e^{-\frac{\alpha\beta}{2}} \operatorname{Vol}(D_r).
\end{align}
Combining \eqref{d1p0} and \eqref{d2p0}, we obtain
\begin{align}\label{eq1p0r}
\operatorname{Vol}(D_r)
\le e^{\frac{\alpha\beta}{2}} r^{\frac{2\alpha a}{\beta}}
\int_{D_r} e^{-\alpha f},
\qquad \text{for all } r \ge 1.
\end{align}
Since on $D_r \setminus D_{r-1}$ one has
\[
\frac{\beta}{2}(r-1)^2 \le f \le \frac{\beta}{2} r^2,
\]
we observe that
\begin{align}\label{eq35p0}
\int_{D_r} e^{-\alpha f} \, dv_g
- \int_{D_{r-1}} e^{-\alpha f} \, dv_g
&= \int_{D_r \setminus D_{r-1}} e^{-\alpha f} \, dv_g \\
\nonumber &\leq e^{-\frac{\alpha\beta}{2}(r-1)^2}
\operatorname{Vol}(D_r \setminus D_{r-1}) \\
\nonumber &\leq e^{-\frac{\alpha\beta}{2}(r-1)^2}
\operatorname{Vol}(D_r) \\
\nonumber &\le e^{-\frac{\alpha\beta}{2}(r-1)^2}
e^{\frac{\alpha\beta}{2}} r^{\frac{2\alpha a}{\beta}}
\int_{D_r} e^{-\alpha f} \, dv_g,
\end{align}
where in the last inequality we used \eqref{eq1p0r}.

Since
\[
e^{-\frac{\alpha\beta}{2}(r-1)^2}
e^{\frac{\alpha\beta}{2}} r^{\frac{2\alpha a}{\beta}} e^r
\longrightarrow 0
\quad \text{as } r \to +\infty,
\]
for $r$ sufficiently large we have
\[
e^{-\frac{\alpha\beta}{2}(r-1)^2}
e^{\frac{\alpha\beta}{2}} r^{\frac{2\alpha a}{\beta}}
\le e^{-r}.
\]
Therefore, by \eqref{eq35p0},
\[
\int_{D_r} e^{-\alpha f} \, dv_g
\le \frac{1}{1 - e^{-r}}
\int_{D_{r-1}} e^{-\alpha f} \, dv_g.
\]

Hence, for any $N \in \mathbb{N}$ and $r$ sufficiently large, we obtain
\[
\int_{D_{r+N}} e^{-\alpha f} \, dv_g
\le \left( \prod_{i=0}^{N}
\frac{1}{1 - e^{-(r+i)}} \right)
\int_{D_{r-1}} e^{-\alpha f} \, dv_g.
\]
Observing that
\[
\prod_{i=0}^{\infty} \frac{1}{1 - e^{-(r+i)}}
\]
converges to a positive real number, and letting $N \to +\infty$, we conclude that
\begin{align}\label{eq37p0}
\int_M e^{-\alpha f} \, dv_g < +\infty.
\end{align}
This proves (i).

Finally, combining \eqref{eq1p0r} and \eqref{eq37p0}, we obtain
\[
\operatorname{Vol}(D_r)
\le C r^{\frac{2\alpha a}{\beta}},
\]
where
\[
C = e^{\frac{\alpha\beta}{2}} \int_M e^{-\alpha f} \, dv_g.
\]
\end{proof}

\section{A Volume Estimate}\label{s4}

In this section, we prove Theorem~\ref{Tp2}. We begin by establishing
some auxiliary identities for the operators $L_k$ and $L_{k,f}$ that
will be used in the proof.

\subsection{Auxiliary Identities}\label{sai}

The following identity is an immediate consequence of Lemma~\ref{lema52}.
\begin{proposition}
The operator $L_k$ applied to the position vector
$x=(x_1,\dots,x_{n+1})$ satisfies
\begin{align}\label{l1xi}
L_k x=-(k+1)\sigma_{k+1}\nu,
\end{align}
where $\nu$ is the chosen unit normal field. 
\end{proposition}

\begin{proof}
Fix $E_i$ as the $i$-th canonical basis vector of $\mathbb{R}^{n+1}$.
Applying Lemma~\ref{lema52} with $a=E_i$ gives
\[
L_k x_i=-(k+1)\sigma_{k+1}\nu_i,
\]
where $\nu_i = \langle \nu, E_i \rangle$. Since this holds for every $i=1,\ldots,n+1$, the result follows.
\end{proof}

By a direct computation, we obtain
\begin{align}\label{l1v2}
L_k \varphi^2
= 2\varphi L_k \varphi
+ 2\langle \nabla \varphi, P_k(\nabla \varphi)\rangle.
\end{align}

Fix $p\in\Sigma$ and choose an orthonormal basis $\{e_1,\ldots,e_n\}$ of $T_p\Sigma$ diagonalizing $P_k$.
Since $\nabla x_i \in T\Sigma$ for each $1 \le i \le n+1$, we have
\[
\nabla x_i = \sum_{m=1}^n \langle e_m, \nabla x_i \rangle e_m.
\]
Since $P_k(e_l)=\kappa_l e_l$, where $\kappa_1,\dots,\kappa_n$ are the eigenvalues of
$P_k$, we obtain
\begin{align}\label{trPr}
\sum_{i=1}^{n+1} \langle \nabla x_i, P_k(\nabla x_i)\rangle
&= \sum_{i=1}^{n+1} \sum_{m,l=1}^n
\langle e_m, \nabla x_i\rangle
\langle e_l, \nabla x_i\rangle
\langle e_m, P_k(e_l)\rangle \nonumber\\
&= \sum_{i=1}^{n+1} \sum_{l=1}^n
\langle e_l, \nabla x_i\rangle^2 \kappa_l \nonumber\\
&= \sum_{l=1}^n \kappa_l
\sum_{i=1}^{n+1} \langle e_l, \overline{\nabla} x_i\rangle^2
= \sum_{l=1}^n \kappa_l
= \operatorname{tr}(P_k),
\end{align}
where the penultimate equality follows from the fact that
$\{\overline{\nabla} x_1,\dots,\overline{\nabla} x_{n+1}\}$ is the canonical basis
of $\mathbb{R}^{n+1}$ and the vectors $e_l$ are unit.

The following proposition provides a clean expression for the operator $L_k$
applied to $f=\frac{|x|^2}{4}$.

\begin{proposition}\label{l1f}
For $f=\frac{|x|^2}{4}$, we have
\[
L_k f
= - (k+1)\sigma_{k+1} \langle \overline{\nabla} f, \nu \rangle
+ \frac{n-k}{2}\sigma_k.
\]
In particular, if $\bm{H}_{k,f}=0$, then
\[
L_k f
= \frac{n-k-2\big|(\overline{\nabla} f)^\perp\big|^2}{2}\,\sigma_k.
\]
\end{proposition}

\begin{proof}
By \eqref{l1v2}, we have
\[
\begin{aligned}
L_k f
&= \frac{1}{4} L_k |x|^2
= \frac{1}{4} \sum_{i=1}^{n+1} L_k x_i^2 \\
&= \frac{1}{2} \sum_{i=1}^{n+1}
\bigl(x_i L_k x_i + \langle \nabla x_i, P_k(\nabla x_i)\rangle\bigr) \\
&= -(k+1)\sigma_{k+1}\langle \overline{\nabla} f,\nu \rangle
+ \frac{1}{2}\sum_{i=1}^{n+1}
\langle \nabla x_i, P_k(\nabla x_i)\rangle.
\end{aligned}
\]
By \eqref{trPr} and Lemma~\ref{l1bc}, we have
\[
\sum_{i=1}^{n+1} \langle \nabla x_i, P_k(\nabla x_i)\rangle
= \operatorname{tr}(P_k)
= (n-k)\sigma_k,
\]
which proves the first statement. If $\bm{H}_{k,f}=0$, then
\[
L_k f
= -\sigma_k\big|(\overline{\nabla} f)^\perp\big|^2
+ \frac{n-k}{2}\sigma_k
= \frac{n-k-2\big|(\overline{\nabla} f)^\perp\big|^2}{2}\sigma_k.
\]
\end{proof}
As a consequence of Proposition~\ref{l1f}, we obtain the following rigidity result.
\begin{cor}\label{cs}
Let $\Sigma^n$ be a compact hypersurface without boundary immersed in
$\mathbb{R}^{n+1}$. Assume that $\Sigma$ is contained in the ball
$\overline{\bm{B}(\sqrt{2(n-k)})}$, that $\sigma_k>0$, and that
$\bm{H}_{k,f}=0$ on $\Sigma$, where $f=\frac{|x|^2}{4}$. Then
$\Sigma$ is the sphere $\mathbb{S}^n(\sqrt{2(n-k)})$.
\end{cor}

\begin{proof}
By the previous proposition and the definition of $L_k$, we have
\begin{align*}
\operatorname{div}(P_k(\nabla f))
&= L_k f
= \frac{n-k-2\big|(\overline{\nabla} f)^\perp\big|^2}{2}\sigma_k \\
&\ge \frac{n-k-2|\overline{\nabla} f|^2}{2}\sigma_k
= \frac{2(n-k)-|x|^2}{4}\sigma_k.
\end{align*}
By the Divergence Theorem, it follows that
\[
0 = \int_\Sigma \operatorname{div}(P_k(\nabla f))\, dv_g \geq \frac{1}{4}\int_\Sigma (2(n-k)-|x|^2)\sigma_k\,dv_g \ge 0.
\]
Hence,
\[
|x|^2 = 2(n-k)\quad \text{on } \Sigma.
\]
Since $\Sigma$ is an immersed hypersurface contained in
$\mathbb{S}^n\bigl(\sqrt{2(n-k)}\bigr)$ and has the same dimension as the
sphere, $\Sigma$ is open in
$\mathbb{S}^n\bigl(\sqrt{2(n-k)}\bigr)$. On the other hand, the compactness
of $\Sigma$ implies that it is closed in
$\mathbb{S}^n\bigl(\sqrt{2(n-k)}\bigr)$. Since $\mathbb{S}^n\bigl(\sqrt{2(n-k)}\bigr)$ is connected, we conclude that $\Sigma = \mathbb{S}^n\bigl(\sqrt{2(n-k)}\bigr)$.

\end{proof}

Before proving Theorem~\ref{Tp2}, we derive a simplified expression for
$L_{k,f}f$.

\begin{proposition}\label{Lrff}
For $f=\frac{|x|^2}{4}$, we have
\[
L_{k,f} f
= -\sigma_k f
+ \langle \overline{\nabla} f, \bm{H}_{k,f} \rangle
+ \langle \nabla f, AP_{k-1}(\nabla f)\rangle
+ \frac{n-k}{2}\sigma_k.
\]
\end{proposition}
\begin{proof}
By the definition of $L_{k,f}$,
\[
L_{k,f}f
=
L_kf-\langle P_k(\nabla f),\nabla f\rangle.
\]
Since $P_k=\sigma_k I-AP_{k-1}$, we have
\[
-\langle P_k(\nabla f),\nabla f\rangle
=
-\sigma_k|\nabla f|^2
+
\langle\nabla f,AP_{k-1}(\nabla f)\rangle.
\]
Moreover, since $f=|\overline{\nabla}f|^2 = |\nabla f|^2 + \bigl|(\overline{\nabla}f)^\perp\bigr|^2$ and
\[
\langle\overline{\nabla}f,\bm{H}_{k,f}\rangle
=
\sigma_k
\bigl|(\overline{\nabla}f)^\perp\bigr|^2
-
(k+1)\sigma_{k+1}
\langle\overline{\nabla}f,\nu\rangle,
\]
it follows from Proposition~\ref{l1f} that
\[
L_{k,f} f
=
-\sigma_k f
+
\langle\overline{\nabla}f,\bm{H}_{k,f}\rangle
+
\langle\nabla f,AP_{k-1}(\nabla f)\rangle
+
\frac{n-k}{2}\sigma_k.
\]
\end{proof}

\subsection{Proof of Theorem \ref{Tp2}}\label{proofTp2}

\begin{theorem}[Theorem~\ref{Tp2}]
Let $\Sigma^n$ be an orientable hypersurface properly immersed in $\mathbb{R}^{n+1}$ such that $\big|(\overline{\nabla}f)^\perp\big||\bm{H}_{k,f}|$ is bounded, where $f=\frac{|x|^2}{4}$.
Assume that $\sigma_k$ is bounded and that $P_k$ is coercive. Then 
\[
\displaystyle\int_{\Sigma} e^{-f} dv_g < +\infty
\]
and $\Sigma$ has polynomial volume growth. More precisely, there exist constants $\beta>0$ and $a \geq 0$ such that, for $r \geq \sqrt{2\beta}$,
\[
\vol(B_r(0) \cap \Sigma) \le C r^{\frac{2 a}{\beta}},
\]
where $C = e^{\frac{\beta}{2}}
(2\beta)^{-\frac{a}{\beta}}
\int_{\Sigma} e^{-f}\,dv_g$.
\end{theorem}

\begin{proof}
Since $P_k$ is coercive, there exists a constant $\delta>0$ such that
\[
\langle P_k(X),X\rangle\geq \delta |X|^2
\]
for every $X\in T\Sigma$. Moreover, by Remark \ref{cisk}, $\sigma_k\geq\delta$. Therefore,
\begin{align*}
\sigma_k f-\langle \nabla f,AP_{k-1}(\nabla f)\rangle
&=\sigma_k\left(|\nabla f|^2+\left|(\overline{\nabla}f)^\perp\right|^2\right)
-\langle \nabla f,AP_{k-1}(\nabla f)\rangle \\
&=\sigma_k\left|(\overline{\nabla}f)^\perp\right|^2
+\langle P_k(\nabla f),\nabla f\rangle \\
&\geq \delta \left|(\overline{\nabla}f)^\perp\right|^2
+\delta|\nabla f|^2 \\
&= \beta\left(
\left|(\overline{\nabla}f)^\perp\right|^2+|\nabla f|^2
\right)
=\beta f\geq0,
\end{align*}
where $\beta:= \delta>0$. Here, we used the identities $f=|\overline{\nabla}f|^2
=|\nabla f|^2+\left|(\overline{\nabla}f)^\perp\right|^2$.

By Proposition~\ref{Lrff}, we have
\[
\begin{aligned}
L_{k,f}f + \sigma_kf - \langle \nabla f, AP_{k-1}(\nabla f)\rangle &= \langle \overline{\nabla} f, \bm{H}_{k,f} \rangle  + \frac{n-k}{2} \sigma_k \\
&\leq \left|\left(\overline{\nabla} f\right)^\perp\right| \left|\bm{H}_{k,f}\right| + \frac{n-k}{2} \sigma_k.
\end{aligned}
\]
Hence,
\[
L_{k,f}f + \beta f \le a,
\]
where 
\[
a := \sup_\Sigma
\left(\big|(\overline{\nabla} f)^\perp\big|\,|\bm{H}_{k,f}|
+ \frac{n-k}{2}\sigma_k \right) < +\infty,
\]
since $\big|(\overline{\nabla} f)^\perp\big|\,|\bm{H}_{k,f}|$ and $\sigma_k$ are bounded.

By Proposition~\ref{l1f}, we have 
\begin{align*}
L_k f  &= -\langle \overline{\nabla} f, (k+1)\sigma_{k+1} \nu \rangle + \frac{n-k}{2}\sigma_k \\
&= -\langle \overline{\nabla} f, \sigma_k \left(\overline{\nabla} f\right)^\perp - \bm{H}_{k,f} \rangle + \frac{n-k}{2}\sigma_k \\
&= -\sigma_k \left|\left(\overline{\nabla} f\right)^\perp \right|^2 + \langle \overline{\nabla} f, \bm{H}_{k,f} \rangle +  \frac{n-k}{2}\sigma_k \\
&\leq a,
\end{align*}
because $\sigma_k > 0$.

We now apply Theorem~\ref{Tp1} with $P=P_k$ and $\alpha=1$.
Since the immersion is proper, $\Sigma$ is complete and
$f=\frac{|x|^2}{4}$ is a proper nonnegative function. Moreover,
\[
\langle P_k(\nabla f),\nabla f\rangle
\geq
\delta|\nabla f|^2
\geq 0.
\]
Therefore,
\[
\int_\Sigma e^{-f}\,dv_g<+\infty
\]
and, for every $s\geq1$,
\[
\operatorname{Vol}(D_s)
\leq
e^{\frac{\beta}{2}}
s^{\frac{2a}{\beta}}
\int_\Sigma e^{-f}\,dv_g,
\]
where
\[
D_s
=
\left\{
x\in\Sigma:
2\sqrt{f}\leq\sqrt{2\beta}\,s
\right\}
=
B_{\sqrt{2\beta}\,s}(0)\cap\Sigma.
\]
Setting $r=\sqrt{2\beta}\,s$, we obtain, for $r\geq\sqrt{2\beta}$,
\[
\operatorname{Vol}(B_r(0)\cap\Sigma)
\leq
C r^{\frac{2a}{\beta}},
\]
where $C= e^{\frac{\beta}{2}}
(2\beta)^{-\frac{a}{\beta}}
\int_\Sigma e^{-f}\,dv_g$.
\end{proof}

Since $\bm{H}_{k,f}\equiv0$ for self-shrinkers of the $Q_k$-flow,
Theorem~\ref{Tp2} immediately yields the following consequence.

\begin{cor}
Let $\Sigma^n$ be an orientable self-shrinker of the $Q_k$-flow
properly immersed in $\mathbb{R}^{n+1}$. If $\sigma_k$ is bounded and
$P_k$ is coercive, then $\Sigma$ has finite weighted volume and
polynomial volume growth.
\end{cor}

When $k=n-1$, the boundedness of $\sigma_{n-1}$ together with the coercivity of $P_{n-1}$ yields an even stronger conclusion.
\begin{proposition}
Let $\Sigma^n$, $n\geq2$, be a complete orientable hypersurface
immersed in $\mathbb{R}^{n+1}$. Assume that $\sigma_{n-1}$ is bounded
and that $P_{n-1}$ is coercive. Then $\Sigma$ is compact.
\end{proposition}

\begin{proof}
Let $\lambda_1,\ldots,\lambda_n$ be the principal curvatures of
$\Sigma$, and let $\{e_1,\ldots,e_n\}$ be a principal frame. Denote by
$\mu_i$ the eigenvalue of $P_{n-1}$ corresponding to $e_i$. By
item~(d) of Lemma~\ref{l1bc},
\[
P_{n-1}e_i
=
\sigma_{n-1}(A_i)e_i=\mu_i e_i,
\]
where $A_i$ denotes the restriction of $A$ to the orthogonal complement
of $e_i$. Since $A_i$ has eigenvalues $\lambda_1,\ldots,\widehat{\lambda_i},\ldots,\lambda_n$, we have
\[
\mu_i
=
\prod_{j\neq i}\lambda_j.
\]

Since $P_{n-1}$ is coercive, there exists $\delta>0$ such that $\mu_i\geq\delta$ for every $i$. Moreover, by item~(a) of Lemma \ref{l1bc}, $\operatorname{tr}(P_{n-1}) = \sigma_{n-1}$, and hence
\[
\sigma_{n-1}
=
\sum_{i=1}^n\mu_i.
\]
Since $\sigma_{n-1}$ is bounded, there exists $M>0$ such that $\delta\leq\mu_i\leq M$ for every $i$.
Now,
\[
\prod_{i=1}^n\mu_i
=
\prod_{i=1}^n\prod_{j\neq i}\lambda_j
=
\left(\prod_{j=1}^n\lambda_j\right)^{n-1},
\]
because each principal curvature $\lambda_j$ occurs in exactly $n-1$
of the products $\mu_i$. Hence
\[
\left|\prod_{j=1}^n\lambda_j\right|
=
\left(\prod_{i=1}^n\mu_i\right)^{\frac{1}{n-1}}
\geq
\delta^{\frac{n}{n-1}}.
\]
Since $\mu_i = \frac{\prod_{j=1}^n\lambda_j}{\lambda_i}$,
we obtain
\[
|\lambda_i|
=
\frac{\left|\prod_{j=1}^n\lambda_j\right|}{\mu_i}
\geq
\frac{\delta^{\frac{n}{n-1}}}{M}
=:c>0.
\]

Furthermore, for $i\neq j$,
\[
\frac{\mu_i}{\mu_j}
=
\frac{\lambda_j}{\lambda_i}>0,
\]
since $\mu_i,\mu_j>0$. Thus all principal curvatures have the same
sign. Consequently,
\[
\lambda_i\lambda_j
=
|\lambda_i|\,|\lambda_j|
\geq c^2
\]
for every $i\neq j$.

By the Gauss equation, $K(e_i,e_j)
=
\lambda_i\lambda_j
\geq c^2$,
and hence
\[
\operatorname{Ric}(e_i,e_i)
=
\sum_{j\neq i}\lambda_i\lambda_j
\geq
(n-1)c^2.
\]
Therefore, $\operatorname{Ric}\geq(n-1)c^2g$.
Since $\Sigma$ is complete, the Bonnet--Myers theorem implies that
$\Sigma$ is compact.
\end{proof}

The proof of Theorem~6 in~\cite{MR4724121} yields the following divergence lemma.
\begin{lemma}\label{divg}
Let $\Sigma^n$ be a hypersurface properly immersed in
$\mathbb{R}^{n+1}$ with polynomial volume growth. Let $T$ be a smooth
tangent vector field on $\Sigma$. Assume that there exist constants
$C>0$ and $d\geq0$ such that
\[
|\operatorname{div} T(x)|
\leq C(1+|x|)^d.
\]
Then
\[
\int_\Sigma
\operatorname{div}\left(e^{-\frac{|x|^2}{4}}T\right)\,dv_g
=0.
\]
\end{lemma}

As a consequence of Theorem~\ref{Tp2} and Lemma~\ref{divg}, we obtain the following result:
\begin{theorem}[Theorem \ref{hiperplan}]
Let $\Sigma^n$ be an orientable self-shrinker of the $Q_k$-flow properly immersed in $\mathbb{R}^{n+1}$. Suppose that $\sigma_k$ is bounded and that $P_k$ is coercive. Let $p\in\mathbb{R}^{n+1}\setminus\{0\}$. If
\[
\int_{\Sigma}
\left\langle
\nabla\left(\frac{\sigma_{k+1}}{\sigma_k}\right),
P_{k-1}(p^\top)
\right\rangle e^{-f}\,dv_g
\leq 0,
\]
then either $\Sigma\not\subset
\left\{
x\in\mathbb{R}^{n+1}:\langle x,p\rangle\geq0
\right\}$, or $k=0$ and $\Sigma
=
\left\{
x\in\mathbb{R}^{n+1}:\langle x,p\rangle=0
\right\}$.

Similarly, if
\[
\int_{\Sigma}
\left\langle
\nabla\left(\frac{\sigma_{k+1}}{\sigma_k}\right),
P_{k-1}(p^\top)
\right\rangle e^{-f}\,dv_g
\geq 0,
\]
then either $\Sigma\not\subset
\left\{
x\in\mathbb{R}^{n+1}:\langle x,p\rangle\leq0
\right\}$, or $k=0$ and $\Sigma
=
\left\{
x\in\mathbb{R}^{n+1}:\langle x,p\rangle=0
\right\}$.

In particular, if $Q_k$ is constant and $k\geq1$, then $\Sigma$ cannot
be contained in either closed half-space determined by a hyperplane
through the origin. If $k=0$, then a self-shrinker contained in one of
these closed half-spaces must coincide with its boundary hyperplane.
\end{theorem}
\begin{proof}
First, assume that
\[
\int_{\Sigma}
\left\langle
\nabla\left(\frac{\sigma_{k+1}}{\sigma_k}\right),
P_{k-1}(p^\top)
\right\rangle
e^{-f}dv_g
\leq 0,
\]
and suppose, by contradiction, that
\[
\langle x,p\rangle\geq 0
\]
on $\Sigma$.

Using the facts that $\operatorname{div}P_k=0$ and
\[
\nabla p^\top=-\langle p,\nu\rangle A,
\]
together with \eqref{eqdivtr} and Lemma \ref{l1bc}, we obtain
\begin{align*}
\operatorname{div}\left(P_k(p^\top)\right)
=-(k+1)\sigma_{k+1}\langle p,\nu\rangle =-\frac{\sigma_k}{2}\langle x,\nu\rangle\langle p,\nu\rangle.
\end{align*}
The second equality follows from the self-shrinker equation
\[
\frac{\sigma_{k+1}}{\sigma_k} = \frac{1}{2(k+1)}\langle x,\nu\rangle.
\]
Since $\sigma_k$ is bounded, it follows that
\[
\left|
\operatorname{div}\left(P_k(p^\top)\right)
\right|
\leq C|x|
\]
for some constant $C>0$.

Next, we compute
\begin{align}\label{eqthp}
\operatorname{div}\left(e^{-f}P_k(p^\top)\right)
\nonumber
&=
e^{-f}
\left[
\operatorname{div}\left(P_k(p^\top)\right)
-\frac{1}{2}\langle x,P_k(p^\top)\rangle
\right] \\
\nonumber
&=
e^{-f}
\left[
-(k+1)\sigma_{k+1}\langle p,\nu\rangle
-\frac{1}{2}\langle x,P_k(p^\top)\rangle
\right] \\
\nonumber
&=
e^{-f}
\left[
-(k+1)\sigma_{k+1}\langle p,\nu\rangle
-\frac{1}{2}\langle P_k(x^\top),p\rangle
\right] \\
\nonumber
&=
e^{-f}
\left\langle
-(k+1)\sigma_{k+1}\nu
-\frac{1}{2}P_k(x^\top),
p
\right\rangle \\
\nonumber
&=
e^{-f}
\left\langle
-\frac{\sigma_k}{2}x^\perp
-\frac{1}{2}P_k(x^\top),
p
\right\rangle \\
\nonumber
&=
e^{-f}
\left\langle
-\frac{\sigma_k}{2}x^\perp
-\frac{\sigma_k}{2}x^\top + \frac{1}{2}AP_{k-1}(x^\top),
p
\right\rangle \\
\nonumber
&=
 -\frac{\sigma_k}{2} \langle x, p  \rangle e^{-f} + \frac{1}{2}\left\langle A(x^\top),
P_{k-1}(p^\top)
\right\rangle e^{-f}\\
&=
-\frac{\sigma_k}{2}\langle x,p\rangle e^{-f}
+
(k+1)
\left\langle
\nabla\left(\frac{\sigma_{k+1}}{\sigma_k}\right),
P_{k-1}(p^\top)
\right\rangle
e^{-f}.
\end{align}
Here, we used the symmetry of $A$ and $P_{k-1}$, the identity
$AP_{k-1}=P_{k-1}A$, the self-shrinker equation, and the identity
\[
\nabla\left(\frac{\sigma_{k+1}}{\sigma_k}\right) = \frac{1}{2(k+1)}A(x^\top).
\]
Indeed, with respect to a local orthonormal frame $\{e_i\}_{i=1}^n$, we have
\begin{align*}
\nabla_{e_i}\left(\frac{\sigma_{k+1}}{\sigma_k}\right)
&=
\frac{1}{2(k+1)}e_i\langle x,\nu\rangle \\
&=
\frac{1}{2(k+1)}
\left(
\left\langle
\overline{\nabla}_{e_i}x,\nu
\right\rangle
+
\left\langle
x,\overline{\nabla}_{e_i}\nu
\right\rangle
\right) \\
&=
\frac{1}{2(k+1)}
\left\langle
A(x^\top),e_i
\right\rangle.
\end{align*}
By Theorem~\ref{Tp2}, $\Sigma$ has polynomial volume growth.
Therefore, integrating \eqref{eqthp} over $\Sigma$ and applying Lemma~\ref{divg} to the tangent vector field
$T=P_k(p^\top)$, we obtain
\begin{align*}
0
&=
\int_{\Sigma}
\operatorname{div}\left(e^{-f}P_k(p^\top)\right) dv_g \\
&=
\int_{\Sigma}
\left[
-\frac{\sigma_k}{2}\langle x,p\rangle
+
(k+1)
\left\langle
\nabla\left(\frac{\sigma_{k+1}}{\sigma_k}\right),
P_{k-1}(p^\top)
\right\rangle
\right]
e^{-f}dv_g. \\
\end{align*}
Since $\sigma_k>0$ by Remark \ref{cisk}, $\langle x,p\rangle\geq0$, and
\[
\int_{\Sigma}
\left\langle
\nabla\left(\frac{\sigma_{k+1}}{\sigma_k}\right),
P_{k-1}(p^\top)
\right\rangle
e^{-f}dv_g
\leq0,
\]
it follows that
\begin{align}\label{eqab}
\int_{\Sigma}
\sigma_k\langle x,p\rangle e^{-f}dv_g =0.
\end{align}
By Remark \ref{cisk}, there exists $\delta>0$ such that $\sigma_k\geq\delta$.
Therefore, \eqref{eqab}, together with $\langle x,p\rangle\geq0$, implies that $\langle x,p\rangle\equiv0$ on $\Sigma$. Hence, the image of $\Sigma$ is contained in the hyperplane $\{x\in\mathbb{R}^{n+1}:\langle x,p\rangle=0\}$. Since $\Sigma$ has the same dimension as this hyperplane, its
image is open in the hyperplane, while properness implies that its
image is closed. Since the hyperplane is connected, $\Sigma$ coincides
with the hyperplane. If $k\geq 1$, this contradicts the lower bound $\sigma_k\geq\delta>0$.

The second assertion follows from the first one by replacing $p$ with
$-p$.
\end{proof}

We now prove Theorem \ref{sphere}.
\begin{theorem}
Let $\Sigma^n$ be an orientable self-shrinker of
the $Q_k$-flow properly immersed in $\mathbb{R}^{n+1}$. Suppose that
$\sigma_k$ is bounded and that $P_k$ is coercive. If
\[
|x^\perp|^2\geq 2(n-k),
\]
then $\Sigma =\mathbb{S}^n(\sqrt{2(n-k)})$.
\end{theorem}

\begin{proof}
By Theorem~\ref{Tp2}, $\Sigma$ has polynomial volume growth. Moreover,
since
\[
\operatorname{div}\bigl(P_k(x^\top)\bigr)
=
\left(
n-k-\frac12|x^\perp|^2
\right)\sigma_k,
\]
and $\sigma_k$ is bounded, we have
\[
\left|
\operatorname{div}\bigl(P_k(x^\top)\bigr)
\right|
\leq C(1+|x|)^2
\]
for some constant $C>0$. Hence Lemma~\ref{divg} applies to
$T=P_k(x^\top)$.

Using the self-shrinker equation, we obtain
\[
\begin{aligned}
\operatorname{div}\left(e^{-f}P_k(x^\top)\right)
&=
\left[
\left(n-k-\frac12|x^\perp|^2\right)\sigma_k
-\frac12
\left\langle P_k(x^\top),x^\top\right\rangle
\right]e^{-f}.
\end{aligned}
\]
Therefore,
\[
0=
\int_\Sigma
\left[
\left(n-k-\frac12|x^\perp|^2\right)\sigma_k
-\frac12
\left\langle P_k(x^\top),x^\top\right\rangle
\right]e^{-f}\,dv_g.
\]
Since $|x^\perp|^2\geq2(n-k)$ on $\Sigma$, $\sigma_k>0$ and $P_k$ is coercive, both terms in
the integrand are nonpositive. Consequently, they vanish identically.
Thus
\[
|x^\perp|^2=2(n-k)
\quad \text{and}\quad
\left\langle P_k(x^\top),x^\top\right\rangle=0.
\]
The coercivity of $P_k$ yields $x^\top\equiv0$.
Hence
\[
|x|^2=|x^\perp|^2=2(n-k),
\]
and therefore $\Sigma$ is contained in
$\mathbb{S}^n(\sqrt{2(n-k)})$. Since the immersion is proper and $\Sigma$ has the same dimension as
$\mathbb{S}^n(\sqrt{2(n-k)})$, its image is both open and closed in the
sphere; hence, by connectedness, $\Sigma=\mathbb{S}^n\left(\sqrt{2(n-k)}\right)$.
\end{proof}

\section{ The Simons cone and the Catenoid}\label{SimonsCone}
 We will describe Examples~\ref{exH} and~\ref{cat}
 mentioned in Remark \ref{rem:comparison-growth-conditions}. 

 Consider the Simons cone
\[
\mathcal{C}
=
C\left(
\mathbb{S}^3\left(\frac{\sqrt{2}}{2}\right)
\times
\mathbb{S}^3\left(\frac{\sqrt{2}}{2}\right)
\right)
=
\left\{
(x',x'')\in\mathbb{R}^4\times\mathbb{R}^4
:
|x'|=|x''|
\right\}
\subset\mathbb{R}^8.
\]
The Simons cone is a nonflat area-minimizing hypercone, smooth away
from its isolated singularity at the origin. Its stability was
established by Simons \cite{MR233295}, while its area-minimizing
property was subsequently proved by Bombieri, De Giorgi, and Giusti
\cite{MR250205}.

Since $\mathcal{C}$ is an area-minimizing hypercone with an isolated
singularity, the theorem of Hardt and Simon \cite{MR809969} yields
smooth area-minimizing hypersurfaces lying on the two sides of
$\mathcal{C}$ and asymptotic to $\mathcal{C}$ at infinity. We consider
a smooth, complete, properly embedded minimal hypersurface
\[
H^7\subset\mathbb{R}^8
\]
lying in the region $\{|x'|>|x''|\}$ and asymptotic to $\mathcal{C}$.

Following the description of $H$ used by Davini \cite{MR2076732}, we choose $H$ to be invariant under the natural
action of $SO(4)\times SO(4)$. Consequently, outside a compact set,
$H$ can be written as a normal graph over $\mathcal{C}$ with radial
height function
\[
x(r,\theta)
=
r\theta+rf(r)\nu_{\mathcal{C}}(\theta),
\]
where $\theta\in\mathcal{C}\cap\mathbb{S}^7$, $r>0$ denotes the radial coordinate on $\mathcal{C}$ and
$\nu_{\mathcal{C}}$ is the unit normal vector field along the regular
part of $\mathcal{C}$ oriented toward the region
$\{|x'|>|x''|\}$.

We are now ready to establish Example \ref{exH}. To this end, set
\[
h(r):=rf(r).
\]
By Proposition~3.3 in \cite{szekelyhidi2026uniqueness} and the observation immediately following its proof, $f(r)=O(r^{-3})$ and $f'(r)=O(r^{-4})$. Hence
\[
h(r)=O(r^{-2}) \quad \text{and}\quad h'(r)=O(r^{-3}).
\]

\begin{ex}\label{exH}
The boundedness of $|x^\perp|$ does not imply sublinear height
growth, even among complete, properly embedded minimal
hypersurfaces.

In fact, outside a compact set, $H$ is parametrized by
\[
x(r,\theta)
=
r\theta+h(r)\nu_{\mathcal C}(\theta),
\]
where $\theta\in\mathcal C\cap\mathbb S^7$. Since $\mathcal C$ is a cone, $\theta$ is tangent to $\mathcal C$,
$\langle\theta,\nu_{\mathcal C}\rangle=0$, and
$\nu_{\mathcal C}$ is invariant in the radial direction. Therefore,
a unit normal vector field along $H$ is
\[
\nu_H
=
\frac{\nu_{\mathcal C}-h'(r)\theta}
     {\sqrt{1+h'(r)^2}}.
\]
Thus, $|x^\perp| \to 0$ as $r\to\infty$. Indeed, 
\[
|x^\perp|
=
|\langle x,\nu_H\rangle|
=
\frac{|h(r)-rh'(r)|}
     {\sqrt{1+h'(r)^2}}
=
O(r^{-2}).
\]
Since $H$ is smooth on the remaining compact subset,
\[
\sup_H |x^\perp|<\infty.
\]

We now show that $H$ does not have sublinearly growing height.
Let
\[
P=\{y\in\mathbb R^8:\langle y,e\rangle=a\},
\qquad |e|=1,
\]
be any affine hyperplane. Since the Simons cone $\mathcal C$ is not
contained in any linear hyperplane, there exists
$\theta_0\in\mathcal C\cap\mathbb S^7$ such that $c:=\langle\theta_0,e\rangle\neq0$. Consider
\[
\gamma(r)
=
r\theta_0+h(r)\nu_{\mathcal C}(\theta_0)
\in H.
\]
Then $\langle\gamma(r),e\rangle-a
=
cr-a+O(r^{-2})$, and hence, for all sufficiently large $r$,
\[
\operatorname{dist}(\gamma(r),P)
\geq \frac{|c|}{2}r.
\]
Since $P$ was arbitrary, no rigid
motion can place $H$ inside a region of sublinear height.
\end{ex}

\begin{ex}\label{cat}
Conversely, sublinear height growth does not imply boundedness of
$|x^\perp|$, even among complete, properly embedded minimal
hypersurfaces.

Consider the classical catenoid
$M \subset\mathbb R^3$, parametrized by
\[
x(s,\theta)
=
\bigl(a\cosh(s)\theta,as\bigr),
\]
where $s\in\mathbb{R}$, $\theta\in\mathbb{S}^1$, $a>0$. Setting $t=a\cosh(s)$, we have $t=\sqrt{x_1^2+x_2^2}$.
Since the height of the catenoid relative to the plane $\{x_3=0\}$
grows logarithmically, it follows in particular that, for every
$\alpha\in(0,1)$,
\[
|x_3|=o(t^\alpha).
\]
Therefore, for all sufficiently large $t$,
\[
|x_3|
\leq
t^\alpha
=
(x_1^2+x_2^2)^{\alpha/2}.
\]
Thus the catenoid has sublinearly growing height.

A choice of unit normal vector field along
$M$ is
\[
\nu(s,\theta)
=
\bigl(\operatorname{sech}(s)\theta,-\tanh(s)\bigr).
\]
Therefore,
\[
\langle x,\nu\rangle
=
a\cosh(s)\operatorname{sech}(s)
-as\tanh(s)
=
a\bigl(1-s\tanh(s)\bigr).
\]
Passing to the limit as $|s| \to+\infty$, we conclude that
\[
\sup_{M}|x^\perp|=\infty.
\]
\end{ex}
We restrict the example to the case $n=2$, since this is the dimension in which the catenoid provides the desired counterexample. Indeed, although the higher-dimensional catenoid is contained in a slab and therefore has sublinear height growth, for $n\geq3$ the quantity $|x^\perp|$ is bounded. This follows directly from the parametrization of the catenoid and the integral representation of $s$ given in \cite{MR2515414}.

\section*{Acknowledgments}
The first author is grateful to Franciele Conrado and Matheus Vieira
for their interest in this work.

\bibliographystyle{plain}
\bibliography{refs}

\end{document}